\documentclass[12pt,oneside]{amsart}
\usepackage{amsmath,amsfonts,latexsym,graphicx,amssymb,url}
\usepackage{hyperref}
\usepackage{txfonts} 
\usepackage{amsmath,txfonts,pifont,bbding,pxfonts,manfnt}
\usepackage{psfrag}
\usepackage{hyperref}
\usepackage{color}
\usepackage[active]{srcltx}
\usepackage{pdfsync}
\numberwithin{equation}{section}
\newcommand{\dist}{\text{dist}} %%distance

\newcommand{\trace}{\text{\rm trace\,}}
\newcommand{\deter}{\text{{\rm det}}}

\newcommand{\Lop}{\mathcal{L}}
\newcommand{\R}{{\mathbb R}} %%reals

\newcommand{\J}{\mathcal{J}}
\newcommand{\la}{\left\langle}
\newcommand{\ra}{\right\rangle}

\def \H {{\mathbb H}}
\def \Hn {{\mathbb H}^n}
\def \e {\epsilon}

\newtheorem{Theorem}{Theorem}[section]

\newtheorem{Lemma}[Theorem]{Lemma}
\newtheorem{Definition}[Theorem]{Definition}

\newtheorem{Remark}[Theorem]{Remark}

\begin{document}

\title[Harnack's inequality for non-divergent equations in $\H^n$]{Harnack's inequality for a class of non-divergent equations in the Heisenberg group}
\author[F. Abedin, C. E. Guti\'errez and G. Tralli]{Farhan Abedin, Cristian E. Guti\'errez and Giulio Tralli} 
\address{Department of Mathematics\\Temple University\\Philadelphia, PA 19122}
\email{tuf28546@temple.edu}
\address{Department of Mathematics\\Temple University\\Philadelphia, PA 19122}
\email{gutierre@temple.edu}
\address{Dipartimento di Matematica \\ Sapienza Universit\`a di Roma\\P.le Aldo Moro 5, 00185 Roma}
\email{tralli@mat.uniroma1.it}
\thanks{C. E. G. was partially supported by NSF grant DMS--1600578}

\begin{abstract}
We prove an invariant Harnack's inequality for operators in non-divergence form structured on Heisenberg vector fields when the coefficient matrix is uniformly positive definite, continuous, and symplectic. The method consists in constructing appropriate barriers to obtain pointwise-to-measure estimates for supersolutions in small balls, and then invoking the axiomatic approach from \cite{FGL} to obtain Harnack's inequality.
\end{abstract}
\maketitle 

\section{Introduction}

In this paper we establish regularity properties of solutions to equations of the form
$$\Lop_A u = \sum\limits_{i,j = 1}^{2n} a_{ij}(z) X_i X_j u = 0, \qquad z \in \Omega \subset \R^{2n+1},$$
where $\Omega$ is an open set, $A(z)=\left(a_{ij}(z)\right)_{i,j}$ is symmetric and uniformly positive definite, and $X_1,\ldots, X_{2n}$ are the Heisenberg vector fields, see Section \ref{sec:preliminaries}.

A challenging problem that researchers have been interested in is to determine the validity of an invariant Harnack's inequality $\sup_B u\leq C\,\inf_B u$ for all non negative solutions $u$,
and for all metric balls $B$, with a constant $C$ depending only on the ellipticity constants of the coefficient matrix $A$. 
When $\mathcal L_A$ is replaced by a standard uniformly elliptic operator with measurable coefficients, that inequality holds and is the celebrated Harnack inequality of Krylov and Safonov \cite{KS80}. Its proof depends in a crucial way upon the maximum principle of Aleksandrov, Bakelman and Pucci, see for example \cite[Section 9.8]{GT83}. 
A number of insightful generalizations and applications of this principle for various elliptic and degenerate-elliptic pdes have been developed, for example, in \cite{CafGut, Cab, DL03, N09, Saf, M14}. In the context of the Heisenberg group, ABP-type maximum principles have been studied in \cite{DGN03, GM04, BCK15}. However, a difficulty to deal with the operator $\mathcal L_A$ is that it is not known if a maximum principle holds true in a form that permits to establish pointwise-to-measure estimates for super solutions such as \cite[Theorem 2.1.1]{Gut01}.
As a result, this precludes one from extending the method of Krylov and Safonov to obtain Harnack's inequality in the present context. 

On the other hand, it was proved in \cite{Gutierrez-Tournier-Harnack}, and extended to more general contexts in \cite{Tralli-Critical-Density}, that when the "contrast" of the coefficient matrix $A$, i.e., the ratio between its maximum and minimum eigenvalues, is sufficiently close to one, then the required pointwise-to-measure estimates for super solutions can be obtained by constructing appropriate barriers, and the Harnack inequality follows from the general theory developed in \cite{FGL}.  

The purpose in this paper is to show that when the matrix coefficient $A$ is uniformly continuous in $\Omega$ and is symplectic (see Definition \ref{def:symplectic matrix}), then it is possible to construct appropriate barriers in a simple and self contained way and obtain the desired critical density estimates for super solutions on sufficiently small balls. 
This yields Harnack's inequality on balls having sufficiently small radius bounded by a constant depending on the modulus of continuity of the matrix $A$, see Theorem \ref{Harnack'sInequality}. 

For non-divergent sub-Riemannian equations with H\"older continuous coefficients, Harnack's inequalities, with constants depending on the H\"older continuity, are proved using parametrix methods in \cite{BU07} for Carnot groups, and in \cite{BBLU} for H\"ormander vector fields.
Theorem \ref{5.2} below yields a stronger result in $\H^1$, since in this case every matrix of unit determinant is symplectic.

Our techniques can easily be used to obtain Harnack's inequality for non-divergent uniformly elliptic operators in $\R^N$ with uniformly continuous coefficients. This is related to a classical result of Serrin \cite{S55}. Serrin's proof follows the one for harmonic functions via a Poisson formula, and it exploits the explicit knowledge of the Poisson kernels on balls for constant coefficient operators. The Dini-continuity is then used to show that these kernels can be used as barriers for the variable coefficient operators. In the setting of degenerate elliptic operators, the notion of the Poisson kernel is much more delicate (see \cite{LU97}), and it is unclear how to proceed with Serrin's strategy. On the other hand, the fundamental solution for the Heisenberg Laplacian in $\H^n$ is well known \cite{Folland}, and the barriers we construct to obtain the critical density estimates are modeled after these special solutions. In fact, in $\H^1$, the fundamental solution of every constant coefficient operator can be written explicitly via a change of basis for the vector fields. A similar change of basis in $\H^n$ for $n\geq 2$ is applicable if we are dealing with a symplectic matrix (see Remark \ref{remsolfond} below), but not with generic matrices (this is related to the fact that the Lie algebra of $\H^n$ is not free when $n \geq 2$, see \cite{BU04, BU05}).

An outline of the paper is as follows. Section \ref{sec:preliminaries} contains preliminaries on the Heisenberg group and the operators considered.
In Section \ref{sec:main lemmas} we include the backbone of our results. Lemma 
\ref{identitylemma} contains an important identity, valid for operators with constant symplectic matrix coefficients, that is essential for the proof of Lemma \ref{subsolutionlemma}.
This leads to the fundamental Lemma \ref{lemmabar} where the desired barrier is constructed, and later used in Section \ref{sec:critical density} to prove the critical density estimates.
Finally, Section \ref{sec:harnack inequality} contains the Harnack inequality and H\"older estimates.

\section{Preliminaries}\label{sec:preliminaries}

We denote coordinates in $\R^{2n+1}$ as $(x,t) = (x_1,\ldots, x_{2n}, t) \in \R^{2n} \times \R$, let $\mathbb{I}_n$ denote the $n \times n$ identity matrix, and define the $2n \times 2n$ matrix
\[\J := \left( \begin{array}{cc}
0 & -\mathbb{I}_n \\
\mathbb{I}_n & 0 \end{array} \right).\] 
The Heisenberg Group $\Hn$ is the homogeneous Lie group $(\R^{2n+1}, \circ, \delta_r)$ equipped with the composition law
$$(x,t) \circ (\xi, \tau) := \left(x + \xi, t + \tau + 2\la \J x, \xi \ra \right)$$
and the family of dilations
$$\delta_{r} : \Hn \rightarrow \Hn, \ \delta_{r}(x,t) = (r x, r^2 t), \ r>0.$$
Here $\la \cdot, \cdot \ra$ is the standard inner product in $\R^{2n}$. The identity element of the group is $0=(0,0)$, and the inverse is  $(x,t)^{-1} := (-x,-t)$.
We consider the homogeneous symmetric norm
$$\rho(x,t) := (|x|^4 + t^2)^{\frac{1}{4}}$$
and its associated distance $d((x,t), (\xi, \tau)) := \rho((x,t)^{-1} \circ (\xi,\tau))$ (c.f. \cite{Cygan}). The balls defined by this distance (called Koranyi balls) will be denoted $B_R((x,t)) := \left\{(\xi,\tau) : d((x,t), (\xi, \tau)) < R \right\}$. For any $(x,t)\in\Hn$ and any $R>0$, we have
$$|B_R((x,t))|=|B_R(0)|=R^Q|B_1(0)|,$$
where $|\cdot|$ denotes the Lebesgue measure. The number
$$Q := 2n+2$$
is the homogeneous dimension of $\Hn$. The Lie algebra of $\Hn$ is generated by the horizontal vector fields
\begin{equation}\label{exi}
X_i = \partial_{x_i} + 2(\J x)_i \partial_t, \ i = 1,\ldots, 2n,
\end{equation}
the horizontal gradient of a function $\psi : \Hn \rightarrow \R$ is
$$\nabla_{\H} \psi := (X_1 \psi, \ldots, X_{2n} \psi),$$
and the horizontal hessian of $\psi$ is the matrix
$$D^2_{\H} \psi := \left(X_{i,j} \psi \right)_{i,j = 1,\ldots, 2n},$$
where $X_{i,j} \psi = \frac{1}{2}\left(X_i X_j \psi + X_j X_i \psi \right)$. To simplify the notation, we will always denote the points
$$z=(x,t),\qquad\zeta=(\xi,\tau).$$

We are concerned with the following class of differential operators
\begin{equation}\label{operator}
\mathcal{L}_A := \text{tr}\left(A(z) D^2_{\H} \ \cdot \right) = \sum\limits_{i,j = 1}^{2n} a_{ij}(z) X_{i,j} = \sum\limits_{i,j = 1}^{2n} a_{ij}(z) X_i X_j
\end{equation}
where $A(z) = (a_{ij}(z))_{i,j = 1,\ldots, 2n} \in \R^{2n \times 2n}$ is symmetric and uniformly  elliptic
\begin{equation}\label{ellipticity}
\lambda \mathbb{I}_{2n} \leq A(z) \leq \Lambda \mathbb{I}_{2n}, \ \text{for all} \ z \in \Omega\subset\Hn,
\end{equation}
with $\Omega$ an open set, and $0 < \lambda \leq \Lambda < +\infty$ fixed constants. Since we  consider only solutions $u$ to the homogeneous equation $\mathcal{L}_A u = 0$, we may assume without loss of generality that $\deter(A(z))=1$ for all $z\in\Omega$. This implies $\lambda\leq 1\leq \Lambda$.
The class of symmetric matrices with unit determinant satisfying \eqref{ellipticity} is denoted by $M_n(\lambda, \Lambda, \Omega)$.

For any constant matrix $M \in M_n(\lambda, \Lambda,\Omega)$ we let
\begin{equation}\label{distancefunction}
\phi_M (x,t) := \la M^{-1} x, x \ra^2  + t^2.
\end{equation}
At times, it will be convenient for us to work with the modified norms $\rho_M:= \phi_M^{1/4}$ and the corresponding modified distance $d_M(z,\zeta):= \rho_M(\zeta^{-1} \circ z)$, which are both one-homogeneous with respect to the dilations $\delta_r$.
It is easy to show that $\rho$ and $\rho_M$ are equivalent:
\begin{equation}\label{equivalentquasidistance}
\sqrt{\dfrac{1}{\Lambda}} \rho \leq \rho_M \leq \sqrt{\dfrac{1}{\lambda}} \rho.
\end{equation}
This implies, in particular, that $d_M$ is a quasi-distance, with constant $\sqrt{\Lambda/\lambda}$ in the quasi-triangular inequality. Moreover, if we denote by $B^M$ the balls with respect to $d_M$, we have
\begin{equation}\label{equivalentquasiballs}
B_{\sqrt{\lambda}r}(z) \subseteq B^M_r(z) \subseteq B_{\sqrt{\Lambda}r}(z) \quad \forall r > 0,\, z \in \H^n.
\end{equation}
Throughout the paper we will denote by $\rm{dist}(\cdot,\cdot)$ and $\rm{diam}(\cdot)$ the distance and the diameter of sets with respect to $d$, whereas we will denote by $\rm{dist}_M(\cdot,\cdot)$ the one with respect to the modified distance $d_M$.

\section{Main Lemmas}\label{sec:main lemmas}

We begin by defining a structural condition on the coefficient matrices that will be necessary to establish our  results, such as the critical density property, and eventually, Harnack's inequality.

\begin{Definition}\label{def:symplectic matrix} 
$A\in M_n(\lambda, \Lambda, \Omega)$ is said to be symplectic if it satisfies the identity 
\begin{equation}\label{symplectic} 
A^{-1}(z) = \J^t A(z) \J
\end{equation}
\noindent at all points $z \in \Omega$.
\end{Definition}

Notice that every symmetric, positive definite $2 \times 2$ matrix with unit determinant is symplectic. 

\vskip 0.3cm

\noindent {\bf Example.} If $A$ is given in block form
\[A(z) = \left( \begin{array}{cc}
A_{11}(z) & A_{12}(z) \\
A_{12}^t(z) & A_{22}(z) \end{array} \right),\] 
\noindent where $A_{11}, A_{22}, A_{12} \in \R^{n \times n}$, $A_{11}, A_{22}$ are symmetric, then $A$ satisfies condition \eqref{symplectic} if and only if the blocks satisfy the identities 
\begin{align*}
 A_{11}(z) A_{22}(z) - A_{12}^2(z) &= \mathbb{I}_n,\\
A_{11}(z)A_{12}^t(z) &= A_{12}(z)A_{11}(z),\\
A_{22}(z)A_{12}(z) &= A_{12}^t(z)A_{22}(z).
\end{align*}
In particular the matrix \[A(z) = \left( \begin{array}{cc}
A_{11}(z) & 0 \\
0 & A^{-1}_{11}(z) \end{array} \right)\]
is symplectic, for any $A_{11}$ symmetric and positive definite.

\vskip 0.3cm

The following lemma establishes some useful identities satisfied by the function $\phi_M$ defined in \eqref{distancefunction} when $M$ is a symplectic matrix with constant entries.

\begin{Lemma}\label{identitylemma} Suppose $M$ is a symmetric, positive definite and  symplectic constant matrix. Then
\begin{equation}\label{usefulidentity}
\dfrac{Q+2}{4} \la M \nabla_{\H} \phi_M, \nabla_{\H} \phi_M \ra = \phi_M \mathcal{L}_M \phi_M=4(Q + 2) \la M^{-1}x,x \ra  \phi_M\quad \forall (x,t)\in\Hn.
\end{equation}
Conversely, if the first identity in  \eqref{usefulidentity} holds for $\phi_M$ in \eqref{distancefunction}, then the matrix $M$ must be symplectic. 
\end{Lemma}
\begin{proof} Direct calculation shows
$$X_j \phi_M = 4 (M^{-1} x)_j \la M^{-1} x, x \ra + 4t (\J x)_j, \ \text{and}$$
$$X_i X_j \phi_M = 4 (M^{-1})_{ji} \la M^{-1} x, x \ra + 8 (M^{-1}x)_i (M^{-1}x)_j +4t \mathcal{\J}_{ji} + 8 (\J x)_i (\J x)_j.$$
Using the antisymmetry of $\J$, we thus have
\begin{equation}\label{Hessian}
X_{i,j} \phi_M = 4(M^{-1})_{ji} \la M^{-1} x, x \ra + 8 (M^{-1}x)_i (M^{-1}x)_j + 8 (\J x)_i (\J x)_j.
\end{equation}
By \eqref{symplectic} we obtain
\begin{align*} 
\la M \nabla_{\H} \phi_M, \nabla_{\H} \phi_M \ra & = \la 4 \la M^{-1}x, x \ra x + 4t M \J x, 4 \la M^{-1}x, x \ra M^{-1}x + 4t \J x \ra \\
& = 16 \la M^{-1} x, x \ra^3 + 16t^2 \la M \J x, \J x \ra \\
& = 16 \la M^{-1} x, x \ra^3 + 16t^2 \la M^{-1} x, x \ra \\
& = 16\la M^{-1} x, x \ra \phi_M(x,t), \ \text{and}
\end{align*}
\begin{align*}
\mathcal{L}_M \phi_M & = \sum\limits_{i,j = 1}^{2n} M_{ij} X_{i,j} \phi_M \\
& = \text{tr} \left(M \left(4\la M^{-1}x,x \ra M^{-1} + 8 (M^{-1}x) \otimes (M^{-1}x) + 8 (\J x) \otimes (\J x)\right) \right) \\
& = 4(Q-2)\la M^{-1}x,x \ra + 8 \la M^{-1}x,x \ra + 8 \la M \J x, \J x \ra \\
& = 4(Q + 2) \la M^{-1}x,x \ra,
\end{align*}
which proves \eqref{usefulidentity}.\\
The converse follows from a review of the previous identities. \end{proof}

\begin{Remark}\label{remsolfond} From \eqref{usefulidentity}, it is easy to see that the function
\begin{equation}\label{fundamentalsolution}
\Gamma_M := \phi_M^{-\frac{Q-2}{4}}
\end{equation}
is, up to a multiplicative constant, the fundamental solution of $\Lop_M$ with pole at $0$. In fact, away from the origin we have
\begin{equation}\label{Gammazero}
\Lop_M \Gamma_M = \frac{Q-2}{4}\phi_M^{-\frac{Q+6}{4}} \left[\dfrac{Q+2}{4} \la M \nabla_X \phi_M, \nabla_X \phi_M \ra - \phi_M \mathcal{L}_M \phi_M \right] = 0.
\end{equation}
\end{Remark}

We now proceed to state precisely the continuity assumptions that are needed on the coefficient matrices $A(z)$ in \eqref{operator}.  In the following, $C(\Omega)$ denotes the set of continuous matrices in $\Omega$, and $||\cdot||$ denotes the operator norm of a matrix. 

\begin{Definition} The modulus of continuity of $A(z)$ at the point $z_0=(x_0, t_0) \in \Omega$ is
$$\omega_A(z_0; \e):=\sup_{z \in B_{\epsilon}(z_0)\cap \Omega}||A(z)-A(z_0)||, \ \epsilon > 0.$$
\end{Definition}

\begin{Definition}\label{equicontinuity} Let $\omega : [0, 1) \rightarrow [0,1)$ be a non-decreasing function satisfying $\lim\limits_{s \rightarrow 0^+} \omega(s) = \omega(0) = 0$. The class $C(\Omega,\omega)$ is the set of matrices $A\in C(\Omega)$ for which $$\omega_A(z_0; \epsilon) \leq \omega(\epsilon)\qquad \mbox{ for all }z_0\in \Omega \ \text{and} \ 0<\e<1. $$
\end{Definition}

\noindent {\bf Example.} Fix $D_0>0$ and consider the class of matrices $A \in C(\Omega)$ satisfying the following uniformly Dini-condition with respect to the distance $d$:
$$\int_0^1{\frac{\omega_A(z_0;\e)}{\e}\,{\rm{d}}\e}\leq D_0\qquad\forall z_0\in\Omega.$$
Then, this class is in $C(\Omega,\omega)$ with  $\omega(\e)=\frac{D_0}{-\log{\e}}$. Indeed, for all $z_0\in\Omega$ and any $0<\e<1$, we have
$$D_0\geq \int_0^1{\frac{\omega_A(z_0;s)}{s}\,{\rm{d}}s}\geq \int_\e^1{\frac{\omega_A(z_0;s)}{s}\,{\rm{d}}s}\geq \omega_A(z_0;s)\int_\e^1{\frac{1}{s}\,{\rm{d}}s}=\frac{\omega_A(z_0;s)}{-\log\e}.$$
A concrete case is the class of $d$-H\"older continuous matrices with exponent $\alpha$, i.e., $\omega_A(z_0;s)\leq C\e^\alpha$; in this setting an invariant Harnack inequality for $\Lop_A$ is proved in \cite{BU07}.

\vskip 0.4cm

In the following lemma we exploit the continuity of the coefficients, the property \eqref{symplectic}, and Lemma \ref{identitylemma}.

\begin{Lemma}\label{subsolutionlemma} Fix $0 < \delta < \frac{1}{2}$. Suppose $A \in M_n(\lambda, \Lambda)$ is continuous at the point  $z_0\in \Omega$, and the matrix $A(z_0)$ is symplectic. Denote $\alpha : = \dfrac{Q-2}{4} + \delta$ and let $M := A(z_0)$. Consider the function
\begin{equation}\label{kernel}
g(\zeta) := -\frac{1}{\alpha} \phi_M^{-\alpha}(\zeta), \qquad\mbox{ for  } \zeta \neq 0.
\end{equation}
There exists $\epsilon_0 > 0$ depending only on $\lambda, \Lambda, Q, \delta$ and $\omega_A(z_0; \cdot)$ such that  
\[
-\text{tr}\left(A(z) (D^2_{\H} g)(\zeta)\right) \geq 0,\quad \ \text{for all} \ z \in B_{\epsilon_0}(z_0)\cap\Omega \text{ and } \zeta \neq 0.
\]
If in addition $A$ is in $C(\Omega, \omega)$, then the choice of $\e_0$ can be made independent of $z_0$ and $A$ (it will depend only on  $\lambda, \Lambda, Q, \delta$ and $\omega$).
\end{Lemma}
\begin{proof} For any $z \in \Omega$, we have for all $\zeta \neq 0$
\begin{align*}
-\text{tr}\left(A(z) (D^2_{\H} g)(\zeta)\right) & = -\sum_{i,j = 1}^{2n} a_{ij}(z) X_i X_j g(\zeta) \\
& = \phi_M^{-\alpha-2} \left\{(\alpha + 1) \la A(z) \nabla_{\H} \phi_M, \nabla_{\H} \phi_M \ra - \phi_M \text{tr}\left(A(z) D^2_{\H} \phi\right) \right\} \\
& = \phi_M^{-\alpha-2} \left\{\left(\dfrac{Q+2}{4} + \delta \right) \la A(z) \nabla_{\H} \phi_M, \nabla_{\H} \phi_M \ra - \phi_M \text{tr}\left(A(z) D^2_{\H} \phi\right) \right\}.
\end{align*}
Consider the expressions
\begin{center}
$\displaystyle{\text{I} := \delta \la A(z) \nabla_{\H}\phi_M, \nabla_{\H} \phi_M \ra},$

$\displaystyle{\text{II} := \dfrac{Q+2}{4} \la A(z) \nabla_{\H} \phi_M, \nabla_{\H} \phi_M \ra - \phi_M \text{tr}\left(A(z) D^2_{\H} \phi\right)}.$
\end{center}
We first estimate $\text{I}$. Since $A,M\in M_n(\lambda,\Lambda)$, we have
\begin{equation}\label{estimateforI}
\text{I} =\delta \la A(z) (\nabla_{\H} \phi_M)(\zeta), (\nabla_{\H} \phi_M)(\zeta) \ra \geq \delta \left(\dfrac{\lambda}{\Lambda} \right) \la M (\nabla_{\H} \phi_M)(\zeta), (\nabla_{\H} \phi_M)(\zeta) \ra.
\end{equation}
To estimate $\text{II}$, we write $A(z) = A(z_0) + [A(z) - A(z_0)] = M + R(z)$. Using \eqref{usefulidentity}, we thus obtain
\begin{align*}
\text{II} & = \dfrac{Q+2}{4} \la A(z) \nabla_{\H} \phi_M, \nabla_{\H} \phi_M \ra - \phi_M \text{tr}\left(A(z) D^2_{\H} \phi\right) \\
& = \dfrac{Q+2}{4} \la (M + R(z)) \nabla_{\H} \phi_M, \nabla_{\H} \phi_M \ra - \phi_M \text{tr}\left((M + R(z)) D^2_{\H} \phi\right) \\
& = \dfrac{Q+2}{4} \la R(z) \nabla_{\H} \phi_M, \nabla_{\H} \phi_M \ra - \phi_M \text{tr}\left(R(z) D^2_{\H} \phi\right)
\end{align*}
Assume now that $d(z, z_0) \leq \epsilon$ for $\epsilon \in(0,1)$ to be determined. Then $||R(z)|| \leq \omega(\epsilon)$, and there exists a positive constant $C_1 = C_1(\lambda,\Lambda,Q)$ such that
\begin{equation}\label{boundforquadraticform}
\bigg| \dfrac{Q+2}{4} \la R(z) (\nabla_{\H} \phi_M)(\zeta), (\nabla_{\H} \phi_M)(\zeta) \ra \bigg| \leq C_1 \omega(\epsilon) \la M (\nabla_{\H} \phi_M)(\zeta), (\nabla_{\H} \phi_M)(\zeta) \ra.
\end{equation}
Also, by using \eqref{Hessian}, we have
\begin{eqnarray*}
\text{tr}\left(R(z) D^2_{\H} \phi_M(\zeta)\right) &=& 4 \la M^{-1} \xi, \xi \ra \text{tr}(M^{-1}R(z)) +\\
&+& 8 \la R(z) M^{-1} \xi, M^{-1} \xi \ra + 8 \la R(z) \J \xi, \J \xi \ra.
\end{eqnarray*}
Since $\|R(z)\| < \omega(\epsilon)$, and $\Lambda^{-1} \leq M^{-1} \leq \lambda^{-1}$, we conclude that there exists a constant $C_2 = C_2(\lambda, \Lambda)$ such that 
$$\bigg|\text{tr}\left(R(z) D^2_{\H} \phi_M(\zeta)\right) \bigg| \leq C_2 \omega(\epsilon) |\xi|^2.$$
Therefore, 
\begin{equation}\label{boundfortrace}
\bigg|\phi_M(\zeta) \text{tr}\left(R(z) D^2_{\H} \phi_M(\zeta)\right) \bigg| \leq C_2 \omega(\epsilon) |\xi|^2 \phi_M(\zeta).
\end{equation}
Now, by \eqref{usefulidentity}, we obtain
$$\la M (\nabla_{\H} \phi_M)(\zeta), (\nabla_{\H} \phi_M)(\zeta) \ra = 16 \la M^{-1} \xi, \xi \ra \phi_M(\zeta) \geq \frac{16}{\Lambda}|\xi|^2 \phi_M(\zeta).$$
In conjunction with \eqref{boundfortrace}, this implies the existence of a constant $C_3 = C_3(\lambda,\Lambda)$ such that
$$\bigg|\phi_M(\zeta) \text{tr}\left(R(z) D^2_{\H}\phi(\zeta)\right) \bigg| \leq C_3 \omega(\epsilon) \la M (\nabla_{\H} \phi_M)(\zeta), (\nabla_{\H} \phi_M)(\zeta) \ra.$$
\noindent With the bounds \eqref{boundforquadraticform} and \eqref{boundfortrace}, we thus conclude there exists some constant $C = C(\lambda,\Lambda,Q)$ such that
\begin{equation}\label{estimateforII}
|\text{II}| \leq C \omega(\epsilon) \la M (\nabla_{\H} \phi_M)(\zeta), (\nabla_{\H} \phi_M)(\zeta) \ra.
\end{equation}
Combining our estimates \eqref{estimateforI} and \eqref{estimateforII} for $\text{I}$ and $\text{II}$ respectively, and by noticing that $\phi_M \geq 0$, we obtain
$$-\text{tr}\left(A(z) (D^2_X g)(\zeta)\right) \geq \phi^{-\alpha-2} \left\{\delta \left(\dfrac{\lambda}{\Lambda} \right) - C \omega(\epsilon) \right\} \la M (\nabla_{\H} \phi_M)(\zeta), (\nabla_{\H} \phi_M)(\zeta) \ra.$$
We now choose $\epsilon = \epsilon_0$, with $\epsilon_0$ such that $\omega_A(z_0; \epsilon_0) \leq \dfrac{\delta \lambda}{C \Lambda}$.
\end{proof}

In preparation for the construction of barriers in the following lemma, we fix $0<\delta<\frac{1}{2}$, a point $z_0 \in \Omega$, and for $A\in M_n(\lambda, \Lambda) \cap C(\Omega, \omega)$ symplectic, we consider as before $M = A(z_0)$; also recall $g$ defined in \eqref{kernel}. For any bounded open set $O\subset\Hn$, we define the function
\begin{equation}\label{defh}
h(z) := \int\limits_{O} g(z^{-1} \circ \zeta) \ d\zeta.
\end{equation}
Also, for $\psi$ a smooth non decreasing function of one variable such that $\psi(s)=1$ for $s\geq 2$ and $\psi(s)=0$ for $s<1$, we define for $\mu > 0$ the function
\begin{equation}\label{defhmu}
h_{\mu}(z) := \int\limits_{O} \psi_{\mu}\left(d_M(z,\zeta)\right) g(z^{-1} \circ \zeta) \ d\zeta,
\end{equation}
where $\psi_\mu(s) :=\psi\left(s/\mu\right)$. The function $h_{\mu}$ is $C^{\infty}$ smooth, and  converges uniformly to $h$ as $\mu \rightarrow 0^+$. In the following we denote
\begin{equation}\label{eta}
\eta:=2\sqrt{\Lambda/\lambda}+1.
\end{equation}

\begin{Lemma}\label{lemmabar} Let $\epsilon_0$ be the constant given in Lemma \ref{subsolutionlemma}. There exists a positive constant $C$ depending only on $\lambda, \Lambda, Q, \delta$ such that for all $0<r\leq \e_0$ and $z_0\in\Omega$ with $B_{\eta r}(z_0) \subseteq \Omega$, and for all open sets $O' \Subset O\subseteq B_{r}(z_0)$, we have 
\begin{equation}\label{boundbarr}
\Lop_A h_{\mu}(z)\geq C r^{-4\delta}\qquad\forall\,z\in O',
\end{equation}
for all $0 < \mu < \min\left\{\dfrac{r}{\sqrt{\lambda}},  \dfrac{\text{dist}(O',\partial O)}{2\sqrt{\Lambda}} \right\}$, and for all $A\in M_n(\lambda, \Lambda) \cap C(\Omega, \omega)$ symplectic.
\end{Lemma} 
\begin{proof} 
Let $g_{\mu}(\zeta) = \psi_{\mu}(\rho_M(\zeta)) g(\zeta)$. Since $d_M$ is symmetric,
$$h_{\mu}(z) := \int\limits_{O} g_{\mu}(z^{-1} \circ \zeta) \ d\zeta = \int\limits_{O} g_{\mu}(\zeta^{-1} \circ z) \ d\zeta.$$
For $z \in B_{r}(z_0)$, by \eqref{equivalentquasiballs} and the hypotheses of the lemma we have that $$B_{r}(z_0) \subset B_{2r}(z) \subseteq B^M_{\frac{2r}{\sqrt{\lambda}}}(z) \subseteq B_{2\sqrt{\frac{\Lambda}{\lambda}}r}(z)\subset B_{\eta r}(z_0) \subset \Omega.$$ In particular, $O \subset  B^M_{2r/\sqrt{\lambda}}(z) \subset \Omega$ for any $z \in O$. By the smoothness of $g_{\mu}$ and the left-invariance of the vector fields and the Lebesgue measure, we have
\begin{align*}
X_i X_j h_{\mu}(z) & = \int\limits_O (X_i X_j g_{\mu})(\zeta^{-1} \circ z) \ d\zeta \\
& = \int\limits_{B^M_{2r/\sqrt{\lambda}}(z)} (X_i X_j g_{\mu})(\zeta^{-1} \circ z) \ d\zeta - \int\limits_{B^M_{2r/\sqrt{\lambda}}(z) \backslash O} (X_i X_j g_{\mu})(\zeta^{-1} \circ z) \ d\zeta \\
& = \int\limits_{B^M_{2r/\sqrt{\lambda}}(0)} (X_i X_j g_{\mu})(\zeta) \ d\zeta - \int\limits_{B^M_{2r/\sqrt{\lambda}}(z) \backslash O} (X_i X_j g_{\mu})(\zeta^{-1} \circ z) \ d\zeta \\
& = \int\limits_{\partial B^M_{2r/\sqrt{\lambda}}(0)} (X_j g_{\mu})(\zeta) \dfrac{X_i \rho_M(\zeta)}{|D \rho_M(\zeta)|} \ d\sigma(\zeta) - \int\limits_{B^M_{2r/\sqrt{\lambda}}(z) \backslash O} (X_i X_j g_{\mu})(\zeta^{-1} \circ z) \ d\zeta,
\end{align*}
where $|D\rho_M|$ stands for the Euclidean length of the standard gradient in $\R^{2n+1}$ and $d \sigma$ is the standard $2n$-dimensional Hausdorff measure in $\R^{2n+1}$. The last step follows from the divergence theorem since the vector fields $X_j$ in \eqref{exi} are divergence-free.
Next we want to replace $g_\mu$ by $g$ in the last two integrals.
If $0<\mu<r/\sqrt{\lambda}$, then for $\zeta\in\partial B^M_{2r/\sqrt{\lambda}}(0)$, $\rho_M(\zeta)>2\mu$ and so $g_{\mu} = g$ in the first integral. 
In the second integral, $g_{\mu}(\zeta^{-1} \circ z) =  g(\zeta^{-1} \circ z)$ if and only if $d_M(\zeta, z)>2\mu$. If $O'$ is a compactly contained subset of $O$, then for $z\in O'$ and $\zeta\notin O$ we have 
$d_M(\zeta,z)\geq \dist_M(O',\partial O)\geq \dfrac{1}{\sqrt{\Lambda}}\dist(O',\partial O)$ by \eqref{equivalentquasidistance}. So if $\mu$ satisfies $0 < 2\sqrt{\Lambda}\mu < \text{dist}(O',\partial O)$, then we can eliminate $\mu$ in the second integral. Therefore, for any $z \in O'$, we obtain
$$X_i X_j h_{\mu}(z) = \int\limits_{\partial B^M_{2r/\sqrt{\lambda}}(0)} (X_j g)(\zeta) \dfrac{X_i \rho_M(\zeta)}{|D \rho_M(\zeta)|} \ d\sigma(\zeta) - \int\limits_{B^M_{2r/\sqrt{\lambda}}(z) \backslash O} (X_i X_j g)(\zeta^{-1} \circ z) \ d\zeta.$$
Multiplying the last identity by $a_{ij}(z)$ and adding over $i,j$ yields
\begin{align*}
(\mathcal{L}_{A} h_{\mu})(z) & = \int\limits_{\partial B^M_{2r/\sqrt{\lambda}}(0)} \dfrac{\la A(z) (\nabla_{\H} g)(\zeta) , (\nabla_{\H} \rho_M)(\zeta) \ra}{|D \rho_M(\zeta)|} \ d\sigma(\zeta) \\
& - \int\limits_{B^M_{2r/\sqrt{\lambda}}(z) \backslash O} \text{tr}\left(A(z) (D^2_{\H} g)(\zeta^{-1} \circ z) \right) \ d\zeta \qquad\forall\,z\in O'.
\end{align*}
Since $z\in O'\subset O\subseteq B_r(z_0)\subseteq B_{\epsilon_0}(z_0) \cap \Omega$ and $A $ satisfies the assumptions of Lemma \ref{subsolutionlemma}, we conclude that 
\begin{equation*}
-\text{tr}\left(A(z) (D^2_{\H} g)(\zeta^{-1} \circ z) \right) \geq 0 \ \text{for all} \ \zeta \in B^M_{2r/\sqrt{\lambda}}(z)\backslash O.
\end{equation*}
Thus,
\begin{eqnarray}\label{rhsQ}
(\mathcal{L}_{A} h_{\mu})(z) & \geq &\int\limits_{\partial B^M_{2r/\sqrt{\lambda}}(0)} \dfrac{\la A(z) (\nabla_{\H} g)(\zeta) , (\nabla_{\H} \rho_M)(\zeta) \ra}{|D \rho_M(\zeta)|} \ d\sigma(\zeta) \nonumber\\
& = & 4\int\limits_{\partial B^M_{2r/\sqrt{\lambda}}(0)} \dfrac{\la A(z) (\nabla_{\H} \rho_M)(\zeta) , (\nabla_{\H} \rho_M)(\zeta) \ra}{\rho_M^{4\alpha+1}(\zeta) |D \rho_M(\zeta)|} \ d\sigma(\zeta) \nonumber\\
& = &4\left(\frac{\sqrt{\lambda}}{2r}\right)^{4\alpha+1}\int\limits_{\partial B^M_{2r/\sqrt{\lambda}}(0)} \dfrac{\la A(z) (\nabla_{\H} \rho_M)(\zeta) , (\nabla_{\H} \rho_M)(\zeta) \ra}{|D \rho_M(\zeta)|} \ d\sigma(\zeta) \nonumber\\
& = &4\left(\frac{\sqrt{\lambda}}{2r}\right)^{Q-1+4\delta}\int\limits_{\partial B^M_{2r/\sqrt{\lambda}}(0)} \dfrac{\la A(z) (\nabla_{\H} \rho_M)(\zeta) , (\nabla_{\H} \rho_M)(\zeta) \ra}{|D \rho_M(\zeta)|} \ d\sigma(\zeta)\nonumber\\
& \geq &\frac{4^{1-2\delta}\lambda^{2\delta}}{r^{4\delta}}\frac{\lambda}{\Lambda}\left(\left(\frac{\sqrt{\lambda}}{2r}\right)^{Q-1}\int\limits_{\partial B^M_{2r/\sqrt{\lambda}}(0)} \dfrac{\la M (\nabla_{\H} \rho_M)(\zeta) , (\nabla_{\H} \rho_M)(\zeta) \ra}{|D \rho_M(\zeta)|} \ d\sigma(\zeta)\right),
\end{eqnarray}
for all $z\in O'$.
Let 
\[
\alpha(M,r)=
\frac{1}{r^{Q-1}}\,\int\limits_{\partial B^M_{r}(0)} \dfrac{\la M (\nabla_{\H} \rho_M)(\zeta) , (\nabla_{\H} \rho_M)(\zeta) \ra}{|D \rho_M(\zeta)|} \ d\sigma(\zeta).
\]
We now adapt an argument from \cite[Section 5.5]{BLU} to show that $\alpha(M,r)$ can be bounded below by a positive constant independent of $r$ and $M$.  
First notice that, for all $r>0$,
$$
\alpha(M,r)=\frac{1}{2-Q}\int\limits_{\partial B^M_r(0)} \frac{\left\langle M\nabla_{\H}\Gamma_M(z),\nabla_{\H}\rho_M(z)\right\rangle}{|D\rho_M(z)|} \ d\sigma(z),$$
where $\Gamma_M$ is as in \eqref{fundamentalsolution}.
By \eqref{Gammazero} and the divergence theorem, we obtain for any $r_2>r_1>0$
\begin{align*}
0 & =\frac{1}{2-Q}\int_{B^M_{r_2}(0)\smallsetminus B^M_{r_1}(0)}{\Lop_M\Gamma_M (z)\ dz}\\
&=\frac{1}{2-Q}\int\limits_{\partial B^M_{r_2}(0)} \left\langle M\nabla_{\H}\Gamma_M(z),\frac{\nabla_{\H}\rho_M(z)}{|D\rho_M(z)|}\right\rangle \ d\sigma(z) - \frac{1}{2-Q} \int\limits_{\partial B^M_{r_1}(0)} \left\langle M\nabla_{\H}\Gamma_M(z),\frac{\nabla_{\H}\rho_M(z)}{|D\rho_M(z)|}\right\rangle \ d\sigma(z) \\
&=
\alpha(M,r_2)-\alpha(M,r_1),
\end{align*}
obtaining that $\alpha(M,r)=\alpha(M,1)$ for all $0<r<\infty$. 
Multiplying the last identity by $r^{Q-1}$, integrating from $0$ to $1$, and using the coarea formula yields
\begin{align*}
\dfrac{1}{Q}\alpha(M,1) & = \int\limits_0^1r^{Q-1} \alpha(M,r) \ dr = \int_0^1 \int_{\partial B^M_r(0)} \frac{\left\langle M\nabla_{\H}\rho_M(z),\nabla_{\H}\rho_M(z)\right\rangle}{|D\rho_M(z)|} d\sigma(z) \ dr \\
& = \int_{B^M_r(0)} \left\langle M\nabla_{\H}\rho_M(z),\nabla_{\H}\rho_M(z)\right\rangle dz\\
&=\frac{1}{16}\int\limits_{B^M_1(0)} \frac{\left\langle M\nabla_{\H}\phi_M(z),\nabla_{\H}\phi_M(z)\right\rangle}{\phi^{\frac{3}{2}}_M(z)} \ dz\qquad \text{by definition of $\phi_M$}\\
&=\int\limits_{B^M_1(0)} \frac{\left\langle M^{-1}x,x\right\rangle}{\phi^{\frac{1}{2}}_M(x,t)} \ dxdt\geq \frac{\lambda}{\Lambda}\int\limits_{B_{\sqrt{\lambda}}(0)} \frac{|x|^2}{\sqrt{|x|^4+t^2}} \ dxdt=:\tilde{C}, \qquad \text{by \eqref{usefulidentity} and \eqref{equivalentquasidistance}.}
\end{align*}
Therefore we obtain that $\alpha(M,r)$ is bounded below uniformly in $r$ and $M$,
and so from \eqref{rhsQ}
$$(\mathcal{L}_{A} h_{\mu})(z)  \geq \frac{4^{1-2\delta}\lambda^{2\delta}}{r^{4\delta}}\frac{\lambda}{\Lambda}\tilde{C}=:\frac{C}{r^{4\delta}}
$$
for all $z\in O'$ and for all $\mu$ satisfying $0 < \mu < \min\left\{\dfrac{r}{\sqrt{\lambda}},  \dfrac{\text{dist}(O',\partial O)}{2\sqrt{\Lambda}} \right\}$. 
\end{proof}

\begin{Remark}
Among all the possible sets $O$ of fixed measure, the set that maximizes the quantity
$$\int_{O}\frac{1}{d^{4\alpha}(z,\zeta)} \ d\zeta$$
is the ball centered at $z$ satisfying $\left|B_1\right|R^Q=\left|B_R(z)\right|=\left|O\right|$ (see \cite[pg. 2112]{Gutierrez-Tournier-Harnack}). We thus have
\begin{eqnarray*}
\int_{O}\frac{1}{d^{4\alpha}(z,\zeta)} \ d\zeta&\leq&\int_{B_R(z)}\frac{1}{d^{4\alpha}(z,\zeta)} \ d\zeta=R^{Q-4\alpha}\int_{B_1(0)}\frac{1}{\rho^{4\alpha}(\zeta)} \ d\zeta=\\
&=&\left| O\right|^{1-\frac{4\alpha}{Q}} \left|B_1\right|^{\frac{4\alpha}{Q}-1}\int_{B_1(0)}\frac{1}{\rho^{4\alpha}(\zeta)} \ d\zeta.
\end{eqnarray*}
Hence, by \eqref{equivalentquasidistance}, we get
\begin{eqnarray*}
0\geq h(z) &=& -\frac{1}{\alpha}\int_{O}\frac{1}{d_M^{4\alpha}(z,\zeta)} \ d\zeta \geq -\frac{\Lambda^{2\alpha}}{\alpha}\int_{O}\frac{1}{d^{4\alpha}(z,\zeta)} \ d\zeta\\
&\geq& -\frac{\Lambda^{2\alpha}}{\alpha}\left| O\right|^{1-\frac{4\alpha}{Q}}\left|B_1\right|^{\frac{4\alpha}{Q}-1}\int_{B_1(0)}\frac{1}{\rho^{4\alpha}(\zeta)} \ d\zeta.
\end{eqnarray*}
Therefore there exists a positive constant $\gamma$ depending only on $Q, \Lambda, \delta$ such that
\begin{equation}\label{lowerb}
0\geq h(z) \geq -\gamma \left| O\right|^{1-\frac{4\alpha}{Q}}.
\end{equation}
\end{Remark}

\section{Critical density}\label{sec:critical density}

We now use the barriers constructed in Lemma \ref{lemmabar} to obtain critical density estimates on balls, first, for balls of radius less than $\epsilon_0$, where $\epsilon_0$ is as in Lemma \ref{subsolutionlemma}, and then for arbitrary balls.
Recall that $\eta$ is as in \eqref{eta}, and $\delta$ is a fixed number in $(0,\frac{1}{2})$.

\begin{Theorem}\label{propsmallr}
There exists $0<\e=\e(Q, \Lambda, \lambda)<1$ such that for all $z_0\in \Omega$ and $0 < r \leq \e_0$ with $B_{\eta r}(z_0)\subseteq \Omega$, for any $A\in M_n(\lambda, \Lambda) \cap C(\Omega, \omega)$ symplectic and for any $u\in C^2(B_{\eta r}(z_0))$ satisfying 
\begin{itemize}
\item[(i)] $u\geq 0$ in $B_{\eta r}(z_0)$,
\item[(ii)] $\Lop_Au\leq 0$ in $B_{\eta r}(z_0)$,
\item[(iii)] $ \inf_{B_{r/2}(z_0)}{u}< \frac{1}{2}$,
\end{itemize}
we have
\begin{equation}\label{criticaldensityestimate}
\left|\{z\in B_{r}(z_0)\,:\,u(z)< 1\}\right|\geq \e\left|B_{r}(z_0)\right|.
\end{equation}
\end{Theorem}
\begin{proof} Let $\varphi(z):=d(z,z_0)^4$. We have
\begin{eqnarray*}
\Lop_A\varphi(z)&=&4\trace(A(z))|x-x_0|^2+8\la A(z)(x-x_0),x-x_0\ra + 8\la A(z)\J(x-x_0),\J(x-x_0)\ra\\
&\leq& 4\Lambda(Q+2)|x-x_0|^2\leq4\Lambda(Q+2)r^2\qquad\mbox{ for any }z\in B_r(z_0).
\end{eqnarray*}
Let $C$ be the constant in \eqref{boundbarr} and consider
$$w(z):=\frac{C r^{2-4\delta}}{4\Lambda(Q+2)}\left(u(z)+ \frac{1}{r^4}\varphi(z)-1\right).$$
By (i), $w$ is nonnegative on $\partial B_{r}(z_0)$. By (iii), there exists a point $\overline{z}\in B_{\frac{r}{2}}(z_0)$ such that $u(\overline z)<1/2$. Therefore, 
\begin{equation}\label{upperb}
w(\overline{z})\leq\frac{Cr^{2-4\delta}}{4\Lambda(Q+2)}\left(\frac{1}{2}+\frac{1}{16}-1\right)=-\frac{7Cr^{2-4\delta}}{64\Lambda(Q+2)}.
\end{equation}
Let $O:=\{z\in B_{r}(z_0)\,:\,w(z)<0\}$. Notice that $O$ is open, $\overline{z} \in O$, and
$$O\subseteq\{z\in B_{r}(z_0)\,:\,u(z)< 1\}.$$ 
With this choice of $O$ and by defining $M := A(z_0)$, we consider the barriers $h, h_\mu$ in \eqref{defh}, \eqref{defhmu} respectively.
We claim that $$h-w\leq 0 \quad\mbox{in }O.$$ 
By definition, $h$ is non-positive. Since $w= 0$ on $\partial O$, it follows that $h - w \leq 0$ on $\partial O$. Suppose, for contradiction, that there exists $\zeta_0\in O$ such that $h(\zeta_0)-w(\zeta_0)=2\sigma>0$. Since $h_\mu$ converges uniformly to $h$ as $\mu$ goes to $0$, there exists $\mu_0>0$ such that $h_\mu(\zeta_0)-w(\zeta_0)\geq\sigma$ for $\mu\leq\mu_0$. Let $O'\Subset O$ containing $\zeta_0$ and $0 < \mu < \min\left\{\dfrac{r}{\sqrt{\lambda}},  \dfrac{\text{dist}(O',\partial O)}{2\sqrt{\Lambda}}, \mu_0 \right\}$. By (ii), $\Lop_A u\leq0$ in $B_{\eta r}(z_0)$, and so $\Lop_A w\leq C r^{-4\delta}$ in $B_{r}(z_0)$. Therefore, by Lemma \ref{lemmabar}, $\Lop_A(h_\mu-w)\geq 0$ in $O'$. From the weak maximum principle for $\Lop_A$ we then infer that $\max\limits_{\partial O'}{(h_\mu-w)}\geq\sigma$. Letting $\mu\rightarrow0^+$, we conclude that $\max\limits_{\partial O'}{(h-w)}\geq\sigma$ for any $O'$ containing $\zeta_0$. This is a contradiction, as $h-w\leq 0$ on $\partial O$. This proves the claim.

Therefore, by \eqref{lowerb}, \eqref{upperb} and recalling that $4\alpha=Q-2+4\delta$, we obtain
$$-\frac{7Cr^{2-4\delta}}{64\Lambda(Q+2)}\geq w(\overline{z})\geq h(\overline{z})\geq -\gamma\left| O\right|^{1-\frac{4\alpha}{Q}}=-\gamma\left| O\right|^{\frac{2}{Q}(1-2\delta)}.$$
This, of course, implies
\begin{eqnarray*}
\left| O\right|^{\frac{2}{Q}(1-2\delta)}&\geq&\frac{C}{\gamma\Lambda}\frac{7}{64(Q+2)}r^{2-4\delta}\\
&=& \frac{C}{\gamma\Lambda}\frac{7}{64(Q+2)}|B_1(0)|^{-\frac{2}{Q}(1-2\delta)}\left| B_r(z_0)\right|^{\frac{2}{Q}(1-2\delta)}=:C_0\left| B_r(z_0)\right|^{\frac{2}{Q}(1-2\delta)}.
\end{eqnarray*}
Choosing $\e=C_0^{\frac{Q}{2(1-2\delta)}}$ therefore gives us
$$\left|\{z\in B_{r}(z_0)\,:\,u(z)< 1\}\right|\geq\left|O\right|\geq\e\left|B_{r}(z_0)\right|.$$
Notice that $\epsilon$ depends only on $Q, \lambda, \Lambda$.
\end{proof}

\begin{Remark}\label{arbitraryball}
We can extend Theorem \ref{propsmallr} to the case where $r>\epsilon_0$ if we assume in addition that $|\Omega|<+\infty$. In this case,  \eqref{criticaldensityestimate} holds with a constant $\epsilon$ depending also on $\omega$ and $|\Omega|$.

Since  $B_{\eta r}(z_0)\subset \Omega$, we have $c_Q\,(\eta\,r)^Q\leq |\Omega|$, where $c_Q = |B_1(0)|$. Hence, $\epsilon_0 < r\leq \dfrac{1}{\eta}\left(\dfrac{|\Omega|}{c_Q} \right)^{1/Q}:=\bar C$.
From (iii) there exists $\bar z\in B_{r/2}(z_0)$ with $u(\bar z)<1$.
We then have that $B_{\epsilon_0/2}(\bar z)\subset B_r(z_0)$ and $B_{\epsilon_0\eta/2}(\bar z)\subset B_{\eta r}(z_0)$. Therefore we can apply Theorem \ref{propsmallr} with $r\leadsto \epsilon_0/2$ and $z_0\leadsto \bar z$ obtaining 
\[
|\{z\in B_{\epsilon_0/2}(\bar z):u(z)<1\}|\geq \epsilon |B_{\epsilon_0/2}(\bar z)|.
\]
Hence
\begin{align*}
|\{z\in B_r(z_0):u(z)<1\}|
&\geq 
|\{z\in B_{\epsilon_0/2}(\bar z):u(z)<1\}|\geq \epsilon |B_{\epsilon_0/2}(\bar z)|\\
&=\epsilon\,\left(\dfrac{\epsilon_0}{2r} \right)^Q\,|B_r(z_0)|\geq
\epsilon\,\left(\dfrac{\epsilon_0}{2\bar C} \right)^Q\,|B_r(z_0)|.
\end{align*}
Combining the above with \eqref{criticaldensityestimate} gives for $0<r<\bar C$
\[
|\{z\in B_r(z_0):u(z)<1\}|
\geq
\bar \epsilon
\,
|B_r(z_0)|,
\]
where $\bar \epsilon= \epsilon\min\left\{1, \,\left(\dfrac{\epsilon_0}{2\bar C} \right)^Q \right\}$.
\end{Remark}

\section{Conclusions and Harnack's inequality}\label{sec:harnack inequality}

The Harnack's inequality now follows from the double ball property (see \cite{Gutierrez-Tournier-Harnack, Tralli-Double-Ball}) and the results in \cite{FGL}. In fact, we have the following Theorem.

\begin{Theorem}\label{Harnack'sInequality} Let $0<\lambda\leq\Lambda$ and let $\omega$ be a function as in Definition \ref{equicontinuity}. There exist constants $C\geq 1$ and $K\geq \eta$ depending only on $\Lambda,\lambda,Q$ and a constant $\epsilon_0$ depending in addition on $\omega$, such that for any $A\in M_n(\lambda, \Lambda) \cap C(\Omega, \omega)$ symplectic and for any $u \in C^2(\Omega)$ satisfying
$$u\geq 0\,\,\,\,\,\mbox{and }\,\,\Lop_A u=0\,\,\,\,\,\mbox{in }B_{K r}(z_0)\subset\Omega \,\,\mbox{ for some }\, r\leq\frac{\eta}{K}\epsilon_0,$$
we have
\begin{equation}\label{harn}
\sup_{B_{r}(z_0)}{u}\leq C\inf_{B_{r}(z_0)}{u}.
\end{equation}
\end{Theorem}

\begin{proof} We use the axiomatic approach to prove Harnack's inequality developed in \cite{FGL} for the doubling quasi metric H\"older space $(\H^n,d,\left|\cdot\right|)$. The reverse doubling and ring conditions are satisfied from the homogeneity. Moreover, if $A\in M_n(\lambda, \Lambda)$, then the family of nonnegative supersolutions of $\Lop_A$ satisfies the double ball property by \cite[Theorem 4.1]{Gutierrez-Tournier-Harnack} (see also \cite{Tralli-Double-Ball}). For this step the continuity of $A$, the symplectic condition, and the restriction on the radius $r$ are unnecessary. In addition, if $A\in  C(\Omega, \omega)$ and is symplectic, then from Theorem \ref{propsmallr} the family
$$\{u\in C^2(V,\R) : V\subset\Omega\cap B_{\eta \epsilon_0}(z_0), u\geq 0 \mbox{ and } \Lop_A u\leq 0 \mbox{ in } V\}$$
satisfies the $\epsilon$-critical density for each $z_0\in\Omega$. In fact, if $B_{\eta r}(\zeta_0)\subseteq \Omega\cap B_{\eta \epsilon_0}(z_0)$, then $r\leq\epsilon_0$ and Theorem \ref{propsmallr} is applicable. 
It follows from \cite[Theorem 4.7, set of conditions (A), and Theorem 5.1]{FGL} that there exist constants $C, K$ depending only on $\lambda, \Lambda, Q$ such that all nonnegative solutions $u$ to $\Lop_A u=0$ in $B_{K r}(z_0)\subset\Omega \cap B_{\eta \epsilon_0}(z_0)$ satisfy \eqref{harn}. If $r\leq\frac{\eta}{K}\epsilon_0$ we just need $B_{K r}(z_0)\subset\Omega$, and the statement is proved. Since the constants in the critical density and double ball properties depend only on $\lambda,\Lambda, Q$, it follows that the constants $C$ and $K$ in the Harnack inequality also depend only on $\lambda,\Lambda,Q$.
\end{proof}
Also, from \cite[Theorem 5.3]{FGL}, it follows that the solutions to $\Lop_A u=0$ are H\"older continuous with the estimate
\begin{equation}\label{Holder}
|u(z) - u(\zeta)| \leq C\, \left(\frac{d(z,\zeta)}{r}\right)^{\alpha}\sup_{B_{r}}{|u|}\qquad\mbox{ for all }z,\zeta\in B_{r/3}
\end{equation}
for $r$ small compared to $\epsilon_0$, where the constants $C$ and $\alpha$ depend only on $Q, \lambda$, and $\Lambda$.

By Remark \ref{arbitraryball}, using similar reasoning as in the proof of Theorem \ref{Harnack'sInequality}, we also have the following.

\begin{Theorem}\label{5.2} Let $0<\lambda\leq\Lambda$, $\omega$ a function as in Definition \ref{equicontinuity}, and $|\Omega|<+\infty$. There exist constants $C, K$ depending only on $\Lambda,\lambda,Q,\omega, |\Omega|$ such that, for any $A\in M_n(\lambda, \Lambda) \cap C(\Omega, \omega)$ symplectic and any $u \in C^2(\Omega)$ satisfying
$$u\geq 0\,\,\,\,\,\mbox{and }\,\,\Lop_A u=0\,\,\,\,\,\mbox{in }B_{K r}(z_0)\subset\Omega,$$
we have
$$\sup_{B_{r}(z_0)}{u}\leq C\inf_{B_{r}(z_0)}{u}.$$
\end{Theorem}

Again, by \cite[Theorem 5.3]{FGL}, we obtain an estimate similar to \eqref{Holder}, with constants $C$ and $\alpha$ depending in addition on $\omega$ and $|\Omega|$.

\end{document}